\documentclass[11pt,a4paper]{article}

\usepackage{amssymb,amsmath,amsthm}
\usepackage[english]{babel}
\usepackage{t1enc}
\usepackage[latin2]{inputenc}
\usepackage{epsfig}
\usepackage{subfigure}

\usepackage{sectsty}
\sectionfont{\large}
\usepackage{footnpag}

\newtheorem{thm}{Theorem}
\newtheorem{cor}[thm]{Corollary}
\newtheorem{defi}[thm]{Definition}
\newtheorem{lem}[thm]{Lemma}

\newtheorem{fact}[thm]{Fact}

\begin{document}

\title{Convex Polygons are Self-Coverable}
\author{Bal\'azs Keszegh\thanks{Research supported by Hungarian National Science Fund (OTKA), under grant PD 108406, NN 102029 (EUROGIGA project GraDR 10-EuroGIGA-OP-003), NK 78439, by the J\'anos Bolyai Research Scholarship of the Hungarian Academy of Sciences and DAAD.}\\
\and D\"om\"ot\"or P\'alv\"olgyi\thanks{Research supported by Hungarian National Science Fund (OTKA), under grant PD 104386 and under grant NN 102029 (EUROGIGA project GraDR 10-EuroGIGA-OP-003) and the J\'anos Bolyai Research Scholarship of the Hungarian Academy of Sciences.}
}


\maketitle

\begin{abstract} We introduce a new notion for geometric families called {\em self-co\-ver\-a\-bi\-li\-ty} and show that homothets of convex polygons are self-coverable.
As a corollary, we obtain several results about coloring point sets such that any member of the family with many points contains all colors.
This is dual (and in some cases equivalent) to the much investigated cover-decomposability problem.
\end{abstract}

\medskip

\section{Introduction}

\begin{defi}
A collection of closed sets $\cal S$ is {\em self-coverable} if there exists a {\em self-coverability function} $f$ such that for any $S \in \cal S$  and for any finite point set $P\subset S, |P|=k$ there exists a subcollection ${\cal S}'\subset {\cal S}$, $|{\cal S'}|\le f(k)$  such that $\cup {\cal S'}=S$ but no point of $P$ is in the interior of an $S'\in {\cal S}'$.
\end{defi}

Note that by definition points of $P$ are only in the exterior or on the boundary of regions from ${\cal S}'$. Also, points outside or on the boundary of $S$ are irrelevant, thus we can and will assume that all points of $P$ are in the interior of $S$.

E.g., it is easy to see that (closed) axis-parallel rectangles are self-coverable with $f(k)=k+1$ and that all discs in the plane (or, in fact, the homothets of any set that is a concave polygon or a set with a smooth boundary) are not self-coverable as already $f(1)$ does not exist.

The motivation to study this notion is the following theorem, which is a generalization of a theorem contained implicitly in Cardinal et al.\ \cite{colorful}.
Note that all logarithms are base two in this paper and every polygon is closed, unless stated otherwise.

\begin{thm}\label{connection}
If $\cal S$ is self-coverable with a monotone 
self-coverability function $f$ and any finite set of points can be colored with two colors such that any member of $\cal S$ with at least $m$ points contains both colors, then any finite set of points can be colored with $k$ colors such that any member of $\cal S$ with $m_k= m (f(m-1))^{\lceil \log k\rceil-1}\le k^d$ points contains all $k$ colors (here $d$ is a constant that depends only on $\cal S$).
\end{thm}

As the version stated here is more general, we include a proof, which is similar to the ideas of \cite{colorful}, which in fact extend the ideas of  \cite{KPmulti}.

Note that this statement is dual to the statement that if $\cal S$ is cover-decomposable to two coverings then it is cover-decomposable into $k$ coverings as well. Unfortunately the dual statement is not equivalent to, nor implied by the above theorem. As we use Theorem \ref{connection} several times, it is useful to introduce some notation for the quantities that appear in it.

Our main results about self-coverability are about homothets of convex polygons where we prove the following.

\begin{thm}\label{thm:polygon}
The family of all homothets of a given convex polygon $C$ is self-coverable with $f(k)\le ck$ where the constant $c$ depends only on $C$.
\end{thm}

\begin{cor}
For a given convex polygon\footnote{Note that in this corollary (and in the followings as well) it does not matter whether the underlying polygon is open or closed as any finite hypergraph system that we obtain by the covering relation can be realized by open polygons if and only if it can be realized by closed polygons.} $C$ there is a constant $d$ such that if any finite set of points can be colored with two colors such that any homothetic copy of $C$ with at least $m$ points contains both colors then any finite set of points can be colored with $k$ colors such that any homothetic copy of $P$ with at least $m_k= k^d$ points contains all $k$ colors.
\end{cor}

For triangles and squares we could even determine the exact value of $f$.

\begin{thm}\label{thm:triangle}
The family of all homothets of a given triangle is self-coverable with $f(k)=2k+1$ and this is sharp.
\end{thm}

Using that for homothets of a given triangle $m\le 12$ by a previous result of ours \cite{KP}, we deduce

\begin{cor}
For a given triangle $T$ any finite set of points can be colored with $k$ colors such that any homothet of\ \ $T$ with at least $m_k= 12\cdot (23)^{\lceil \log k\rceil-1}\le 12\cdot k^{\log 23}=O(k^{4.53})$ points contains all $k$ colors.
\end{cor}

This result improves the previous best upper bound which was approximately $O(k^{7.17})$ in \cite{colorful} which was already a big improvement upon our original upper bound from \cite{KPmulti} which was about $12^{2^k}$. On the other hand, we remark that this latter bound, although it is much worse, works in a more general setting, as it proved the cover-decomposability of octants into many coverings.
Since the first draft of this paper, using a method similar to the one introduced here, 
this was also improved to $O(k^{5.53})$ by Cardinal et al.\ \cite{colorful2}.

\begin{thm}\label{thm:square}
The family of all homothets of a square is self-coverable with $f(k)=2k+2$ and this is sharp.
\end{thm}


Unfortunately the existence of $m$ is an open problem for squares thus we can only deduce an if-then statement in this case:

\begin{cor}
If any finite set of points can be colored with two colors such that any axis-aligned square with at least $m$ points contains both colors then
any finite set of points can be colored with $k$ colors such that any (open or closed) square with at least $m_k= m (2m)^{\lceil \log k\rceil-1}$ points contains all $k$ colors.
\end{cor}

We also show that the constant in Theorem \ref{thm:polygon} cannot depend only on the number of vertices of $P$ as even for a quadrangle it can be arbitrarily big:

\begin{thm}\label{thm:quadrangle}
For every $c$ there exists a quadrangle $Q$ for whose homothets $f(k)\ge ck$.
\end{thm}

In the rest of the paper, we prove Theorem \ref{connection} in Section \ref{secconnection}, then Theorem \ref{thm:triangle} and \ref{thm:square} in Section \ref{sectrsq} and finally Theorem \ref{thm:polygon} and \ref{thm:quadrangle} in Section \ref{secdelaunay}. 

For more results on cover-decomposability and related problems, see 
\cite{PPT11} or \cite{P10}.

\section{Connection to cover-decomposability and proof of Theorem \ref{connection}}\label{secconnection}

\begin{proof}[Proof of Theorem \ref{connection}] Suppose that $\cal S$ is self-coverable with self-coverability function $f$ and any finite set of points $P$ can be colored with two colors such that any member of $\cal S$ with at least $m$ points contains both colors. Now we show by induction on $k$ that any finite set of points can be colored with $k$ colors such that any member of $\cal S$ with at least $m_k=m (f(m-1))^{\lceil \log k\rceil-1}$ points contains all $k$ colors.\\

Suppose we already know the statement for $a$ and $b$, from this we establish it for $k=ab$. Color $P$ with $a$ colors using induction and denote the color classes by $P_1,\ldots, P_a$. Now color each of these color classes with $b$ colors using induction. We claim that this coloring is good for $k=ab$. By contradiction, say $S\in \cal S$ does not contain all colors. This means that for some $1\le i\le a$ we have $|S\cap P_i|\le m_b-1$. Using the self-coverability of $\cal S$, cover $S\setminus P_i$ with $f(m_b-1)$ sets of $\cal S$. Using the monotonicity of $f$, one of these covering sets contains at least $\lceil \frac{m_k}{f(m_b-1)} \rceil$ points of $P$ but no points of $P_i$. This contradicts that our coloring with $a$ colors was good if $\lceil \frac{m_k}{f(m_b-1)} \rceil \ge m_a$.\\

Using the above argument for $b=2$, we can see that $m_k= m (f(m-1))^{\lceil \log k\rceil-1}$ satisfies $m_{2a}\ge m_a f(m-1)$, thus we are done (using the monotonicity of $m_k$ if $k$ is odd). 
\end{proof}

\section{Self-coverability of triangles and squares} \label{sectrsq}

\begin{proof}[Proof of Theorem \ref{thm:triangle}]

We now prove that for the family $\cal T$ of homothets of a given triangle $T$ we have $f(k)=2k+1$.
First by affine transformations we can transform the triangle to any other triangle, thus it is enough to prove the statement for one triangle. 
Further, by homothetic symmetry it is enough to prove self-coverability of one fixed size triangle $T_0$. Thus we can assume that $T$ has the three vertices $(0,0), (2,0), (1,1)$.

First we begin by giving a set $P$ of $k$ points for which $2k+1$ triangles are indeed needed to cover $T_0$. Let the set of points be on a vertical line passing through the vertex $(1,1)$, i.e.\ all points of $P$ have coordinates $(1,y); 0<y<1$.
Let $\epsilon$ be a small positive constant and for each point $(1,y)$ of $P$ assign two dummy points $(1-2\epsilon,y-\epsilon)$ and $(1+2\epsilon,y-\epsilon)$ inside $T_0$ (i.e., we choose $\epsilon$ such that $\epsilon<1-y$ for all points of $P$). Put an additional dummy point at coordinate $(1,y_{max}+\epsilon)$ above the highest point $(1,y_{max})$ of $P$.
It is easy to see that if $\epsilon$ is small enough then any triangle from $\cal T$ contained in $T_0$ and not containing a point of $P$ in its interior can cover at most one dummy point in its interior. As to cover $T_0$ in particular we need to cover all dummy points, so we need at least $2k+1$ triangles. See Figure \ref{fig:constructions}(a) for an illustration.

\begin{figure}
\centering
\includegraphics[width=9cm]{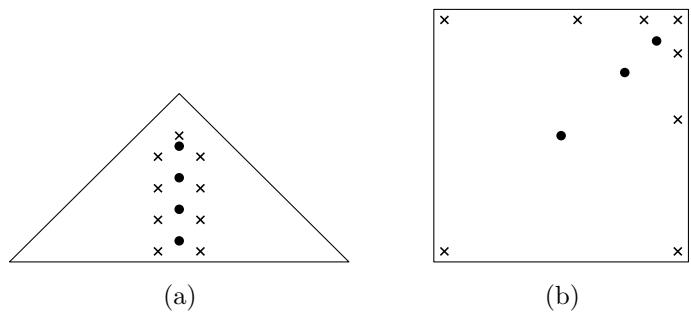}
\caption{Lower bound constructions for self-covering (a) a triangle and (b) a square}
\label{fig:constructions}
\end{figure} 

Now it is enough to prove that at most $2k+1$ triangles are always enough to cover $T_0$. We prove this by induction on $k$. For $k=0$ we can cover $T_0$ by itself. If $k\ge 1$, then take the\footnote{For simplicity we suppose that there is only one such point $p$ yet the proof can be easily modified to the case when there are multiple points with the same $y$-coordinate, in which case we have to handle all these points in one step of the induction.} point $p\in P$ with the smallest $y$-coordinate and denote it by $y(p)$.
Denote 
 by $H_{y}$ the halfplane with the horizontal line $y=y(p)$ as its boundary containing an infinite positive ray on the $y$-axis.
Apply induction on $P\setminus p$ and the triangle $T_1=T_0\cap H_{y(p)}$. See Figure \ref{fig:triangle}
for an illustration. We get a collection ${\cal S}_1$ of at most $2k-1$ triangles (homothetic to $T_0$) covering $T_1$. Denote by ${\cal S}_a$ those triangles from ${\cal S}_1$ whose bottom edge $e$ is on the bottom edge $e_1$ of $T_1$ but $e$ does not contain $p$ in its interior (thus $p$ can be a vertex of such a triangle). 

Now we can scale all triangles in ${\cal S}_a$, their top vertex as the center of scaling, so that their bottom edge goes onto the bottom edge $e_0$ of $T_0$.
Our new collection will consist of these scaled triangles and  ${\cal S}_1\setminus {\cal S}_a$.
The points of $e_1$ not covered by the scaled triangles form an interval and $p$ cuts this interval into a left interval $l$ and a right interval $r$. Now the triangle $T_l$ is (well) defined to be the triangle which intersects $e_1$ in exactly $l$ and has its bottom edge on $e_0$. Similarly, $T_r$ is (well) defined to be the triangle which intersects $e_1$ in exactly $r$ and has its bottom edge on $e_0$. We claim that these two triangles do not contain a point from $P$ in their interior. Indeed, first of all there are no points of $P$ under $H_{y(p)}$. Second, in the inductive construction there must be a triangle $T_l'$ whose bottom edge contains the whole $l$, thus $T_l'$ contains no point of $P$ in its interior and $T_l\cap  H_{y(p)}\subseteq T'_1$. The interior of $T_r$ is similarly disjoint from $P$.

\begin{figure}
\centering
\includegraphics[width=12cm]{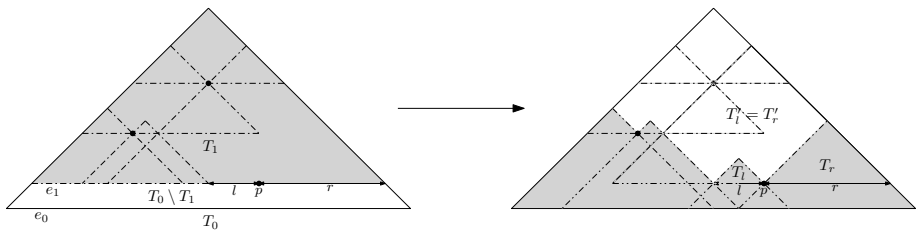}
\caption{Extending a covering of $T_1$ and adding the two triangles $T_l$ and $T_r$}
\label{fig:triangle}
\end{figure} 

We have seen that none of the triangles in this new collection of at most $2k-1+2=2k+1$ triangles contains a point of $P$ in its interior. Now we finish the proof by showing that this collection of triangles covers $T_0$. $T_1$ is trivially covered by induction. For an arbitrary point $q\in T_0 \setminus T_1$ at least one of the two diagonal lines (these are the lines parallel to the non-horizontal edges of $T$) across $q$ intersects $e_1$ in a point $q'$. If $q'$ is on $l$ (or respectively $r$), then $q$ is covered by $T_l$ (or respectively by $T_r$). If none of these happens then $q$ is covered by one of the scaled triangles.
\end{proof}

\begin{proof}[Proof of Theorem \ref{thm:square}]
First we begin by giving a set $P$ of $k$ points for which $2k+2$ squares are indeed needed to cover a square $R=[0,1]\times [0,1]$ if $k\ge 1$. The points are on one of the diagonals of $R$, the $i$th point has coordinates $(1-1/2^i,1-1/2^i)$. Let $\epsilon$ be a small positive constant and for each point $(1-1/2^i,1-1/2^i)$ of $P$ assign two dummy points $(1-1/2^{(i-1)}+i\epsilon,1-\epsilon)$ and $(1-\epsilon,1-1/2^{(i-1)}+i\epsilon)$ inside $R$. Put an additional dummy point at coordinate $(\epsilon,\epsilon)$ and $(1-\epsilon,1-\epsilon)$. It is easy to see that if $\epsilon$ is small enough ($\epsilon<1/2^k(k+1)$ is sufficient) then any square contained in $R$ and not containing a point of $P$ in its interior can cover at most one dummy point. As to cover $R$ in particular we need to cover all dummy points, so we need at least $2k+2$ squares. See Figure \ref{fig:constructions}(b) for an illustration.

Now it is enough to prove that at most $2k+2$ squares are always enough to cover a square. We again proceed by induction but we need a more general statement, Lemma \ref{lem:square}. The theorem follows from this lemma by taking $R$ to be a square.
\renewcommand{\qedsymbol}{}
\end{proof}

The following lemma states that if the ratio of the two sides of an axis-parallel rectangle $R$ is at most $2$ then it can be covered by $2k+2$ axis-parallel squares (while not covering the point set of size $k$), whereas if the ratio of the sides is bigger, then we can cover $R$ such that the squares may hang out over the top edge of $R$ but only until a limited height, and not covering an additional fixed point $p$ on the top edge.  

\begin{lem}\label{lem:square}
Given an axis-parallel rectangle $R$ with width $a$, height $b\le a$ and a point set $P\subset R$, $|P|=k$ 
and a point $p$ on the top edge of $R$, there is a collection $\cal R$ of at most $2k+2$ axis-parallel squares covering $R$, none of them containing a point from $P\cup \{p\}$ in their interior, such that
\begin{itemize}
\item[(i)] if $a/2\le b$ then $\cup {\cal R}=R$,
\item[(ii)] if $b<a/2$ then $R\subseteq\cup {\cal R}\subseteq R\cup R'$, where $R'$ is a rectangle whose bottom side is the top side of $R$
and its height is $b'=a-2b>0$.
\end{itemize}
\end{lem}

Note that $p$ is not in $P$. Also note that in the first case points of $P$ on the boundary of $R$ do not matter while in the second case points of $P$ on the top edge of $R$ are not irrelevant and can modify the choice of squares.
See Figure \ref{fig:square}
for an illustration.

\begin{figure}
\centering
\includegraphics[width=12cm]{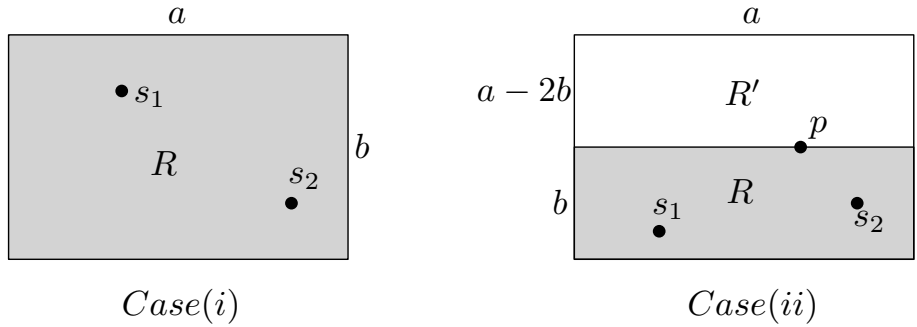}
\label{fig:square}
\caption{the rectangle $R$ to be covered by Lemma \ref{lem:square}, in Case (ii) the squares can cover parts of $R'\setminus R$ as well, but are not allowed to contain $p$ in their interior}
\end{figure} 

\begin{proof}
We can suppose that the bottom left corner of $R$ is the origin $(0,0)$.
We prove the two cases simultaneously by induction on $k$. Both cases will be quite similar, we always cut the rectangle through some point of $P$ into two as equal parts as possible and apply induction on both parts. We denote the $x$- and $y$-coordinate of a point $s$ by $x(s)$ and by $y(s)$, respectively. For a rectangle $Q$ we denote by $int Q$ its interior.

If $k=0$ then in Case (i) it is trivial to cover $R$ using two squares.
In Case (ii) we can suppose without loss of generality that $x(p)\ge a/2$. Now put a square of height $\min(a-b,x(p))$ in the bottom left corner of $R$ and a square of height $max(b,a-x(p))$ in the bottom right corner of $R$, so by definition these do not contain $p$ in their interior. 
It is easy to check that they cover $R$ and are contained in $R'$.

Next suppose $k>0$ and we are in Case (i). If there exists $s\in P$ such that $b/2\le x(s)\le a-b/2$ then cut the rectangle $R$ into two parts $R_1,R_2$ by a vertical line through $s$ and then by induction (Case (i)) we can find squares covering $R_1$ and $R_2$, together they cover $R$ and the number of the squares is at most $2k_1+2+2k_2+2=2(k_1+k_2+1)+2\le 2k+2$, where $k_1=|P\cap int R_1|, k_2=|P\cap int R_2|$.

If there is no such $s$ then choose $s$ to be the point of $P$ which is closest to the vertical halving line of $R$, i.e.\ for which $|x(s)-a/2|$ is minimal. Without loss of generality we can suppose that $x(s)<a/2$ and thus we also know that $x(s)<b/2$. We again cut by the vertical line through $s$. To get a covering of the right rectangle $R_2$ we can apply the induction hypothesis with Case (i). For the left rectangle $R_1$ we apply the induction hypothesis with Case (ii) by setting $p_{1}:=s$, $R_{1}:=R_1$ and $R'_{1}$ being the part of $R$ between the vertical lines at $x$-coordinate $x(s)$ and $b-x(s)$. The two set of squares together cover $R$ and as $b-x(s)<b<a$ implies $R'_{1}\subset R$, they do not hang out from $R$. We need to check if the covering of $R_1$ does not interfere with the points in $P\cap int R_1$. This is true if $int R'_{1}$ does not contain points from $P$, which follows from the fact that there is no point of $P$ with $x$-coordinate between $x(s)$ and $a-x(s)$ and the right edge of $R'_{1}$ has $x$-coordinate $b-x(s)<a-x(s)$. Finally, the number of squares we used is again at most $2k_1+2+2k_2+2=2(k_1+k_2+1)+2\le 2k+2$, where $k_1=|P\cap int R_1|, k_2=|(P\cap R_2)\setminus\{s\}|$.
See Figure \ref{fig:squarecase1} for an illustration. 

\begin{figure}[t]
    \centering
    \subfigure[Case (i), applying induction on the two parts of $R$, Case (ii) on $R_1$ 		and Case (i) on $R_2$]{\label{fig:squarecase1}
        \includegraphics[scale=1.1]{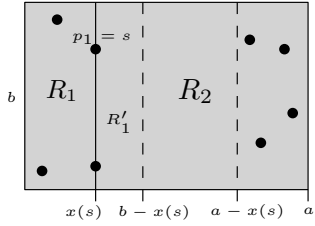}
				}    
    \hskip 20mm
    \subfigure[Case (ii), applying induction on the two parts of $R$, Case (ii) on $R_1$ 		and Case (ii) on $R_2$]{\label{fig:squarecase2}
        \includegraphics[scale=1.1]{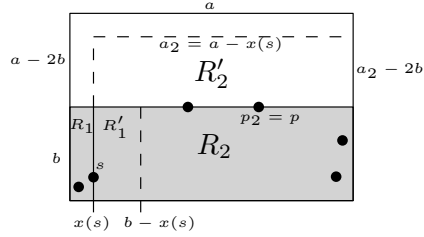}
				}
						\caption{Proof of Lemma \ref{lem:square}}
\end{figure} 	

Suppose now that $k>0$ and we are in Case (ii). Similarly to the previous case, if there exists $s\in P$ such that $b/2\le x(s)\le a-b/2$ then cut the rectangle $R$ into two parts $R_1,R_2$ by a vertical line through $s$ and then by induction (using Case (i) or Case(ii), see details below) we can find squares covering $R_1$ and $R_2$, together they cover $R$ and the number of the squares is at most $2k_1+2+2k_2+2=2(k_1+k_2+1)+2\le 2k+2$, where $k_1=|P\cap int R_1|, k_2=|P\cap int R_2|$. The induction is done in the following way. We just consider $R_2$, $R_1$ can be handled in the same way. If the ratio of the two sides of the rectangle is at most $2$, then we can simply apply induction Case (i). If the ratio of the sides is bigger than $2$, i.e., $b<(a-x(s))/2$, then we apply induction Case (ii) with $p_2=p$ if $p$ is on the top edge of $R_2$ and choosing an arbitrary $p_2$ on the top edge of $R_2$ if $p$ is not on the top edge of $R_2$.

If there is no such $s$ then choose $s$ to be the point of $P$ which is closest to the vertical halving line of $R$, i.e. for which $|x(s)-a/2|$ is minimal. Without loss of generality we can suppose that $x(s)<a/2$ and thus we also know that $x(s)<b/2$. We again cut by the vertical line through $s$. Now in the same way as in Case (i) we can apply induction Case (ii) on the left part $R_1$. It is easy to see that by the choice of $s$, $R'_1$ corresponding to $R_1$ is contained by $R$ and does not contain points from $P$ in its interior.  On the right part $R_2$, again, if the ratio of the two sides $b$ and $a-x(s)$ is at most $2$, then we can simply apply induction Case (i). If the ratio of the sides of $R_2$ is bigger than $2$, i.e., $b<(a-x(s))/2$, then we apply induction Case (ii) on $R_2$ with $p_2=p$ if $p$ is on the top edge of $R_2$ and choosing an arbitrary $p_2$ on the top edge of $R_2$ if $p$ is not on the top edge of $R_2$. It is easy to see that the rectangle $R_2'$ corresponding to $R_2$ is contained in $R'$. Thus, the two set of squares we get by induction again cover the whole $R$, are contained in $R\cup R'$, none of the squares contains $p$ in its interior and the number of squares is at most $2k_1+2+2k_2+2=2(k_1+k_2+1)+2\le 2k+2$, where $k_1=|P\cap R_1|, k_2=|P\cap (R_2\setminus R_1)|$.
See Figure \ref{fig:squarecase2} for an illustration of this case. 
\end{proof}

\section{Self-coverability of convex polygons using Delaunay triangulation} \label{secdelaunay}
In this section we prove Theorem \ref{thm:polygon}. Since the proof is a little complicated, to illustrate it, first we will reprove Theorem \ref{thm:triangle}, then Theorem \ref{thm:square} (with a worse self-coverability function) and only then prove Theorem \ref{thm:polygon} in its full generality.

The proof uses the notion of generalized Delaunay triangulation, which are the dual of generalized Voronoi diagrams.
In the generalized Delaunay triangulation (with respect to some convex shape $C$) of a point set $P$, two points of $P$ are connected if they are covered by a homothet of $C$ not containing any point of $P$ in its interior.
We will use the following properties of such triangulations (see e.g.,\ \cite{K}).

\begin{fact}\label{dfacts}
The generalized Delaunay triangulation with respect to some convex shape $C$ is a connected graph which is unique if the points are in general position with respect to $C$ (meaning no four points fall on the boundary of a homothet of $C$).
The inner faces of this graph are triangles and each inner face is covered by a homothet of $C$ not containing any points in its interior.
\end{fact}

For illustrations of the following proofs, see Figure \ref{fig:delaunaytriangle} and Figure \ref{fig:delaunaysquare}.

\begin{figure}[t]
    \centering
		    \subfigure[Proof of Theorem \ref{thm:triangle}]{\label{fig:delaunaytriangle}
				\includegraphics[scale=0.5]{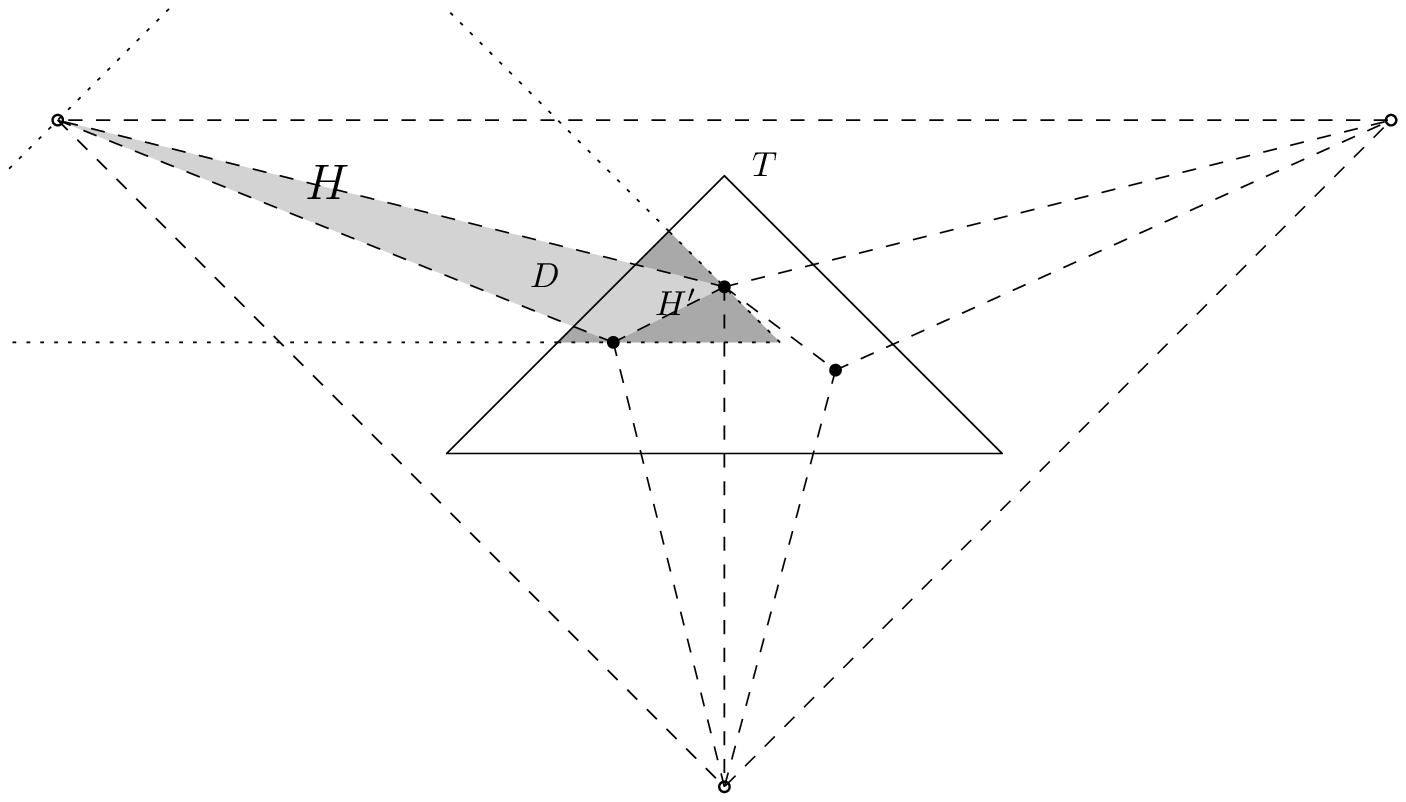}
				}    
				\hskip 10mm
    \centering
				\subfigure[Proof of Theorem \ref{thm:square}]{\label{fig:delaunaysquare}
				\includegraphics[scale=0.65]{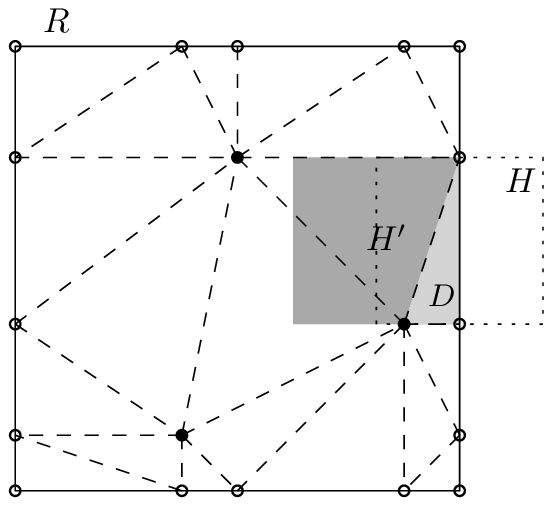}
				}    
				\caption{Proof of Theorem \ref{thm:triangle} and Theorem \ref{thm:square} using Delaunay-triangulations}
\end{figure} 	

\begin{proof}[Second proof of Theorem \ref{thm:triangle}] We add three new points to $P$ which are far, outside of $T$, and form a reflected copy of $T$. Denote the new point set by $P'$. In the\footnote{Since it might not be unique, we should say {\em a} Delaunay triangulation but to avoid confusion we fix one of the Delaunay triangulations if there are more and use the definite article.} Delaunay triangulation determined by $T$, these three points will be all connected, making all the faces triangles. Using Euler's formula, there are $k+3$ vertices and thus $2(k+3)-4$ faces, so we have $2k+1$ inner faces, all of which can be covered by a homothet of $T$ not containing any point of $P'$ in its interior.

The only problem is that these homothets might extend beyond the boundary of $T$. But it is easy to see that for any homothet $H$ of $T$ the triangle $H'=H\cap T$ is also a homothet of $T$, so these give at most $2k+1$ covering triangles.
\end{proof}

\begin{proof}[Second proof of Theorem \ref{thm:square}] (with worse self-coverability function) Similarly to the proof for triangles, we add a few points to $P$ and we denote the new point set by $P'$. Now all new points will be on the boundary of the square $R$. The new points are obtained as follows. For each $p\in P$ project it orthogonally to all four sides of $R$ and add it to $P'$. Also we add the four corners of $R$ to $P'$. Thus $|P'|=5k+4$ if all vertices of $P$ have different coordinates, which we suppose from now on for simplicity. 
In the Delaunay triangulation determined by $R$, all boundary points will be connected to their neighbors (on the boundary), making all inner faces triangles. The infinite outer face has $4k+4$ vertices.
Using Euler's formula, there are $5k+4$ vertices and thus $2\cdot (5k+4)-4-(4k+1)=6k+3$ faces. The $6k+2$ triangular inner faces partition $R$ and all of them can be covered by a homothet of $R$ not containing any point of $P'$ in its interior.

Again, the problem is that these homothets might extend beyond the boundary of $R$.
To take care of this, observe first that each square $H$ that extends beyond $R$ intersects only one side $s$ of $R$.
For each such $H$, push $H$ perpendicularly to $s$ until its outer side overlaps with $s$, call the pushed square $H'$.
This way no point $p\in P$ can get into the interior of $H'$ since then the projection of $p$ to $s$ would have been also inside $H$.
So we get $f(k)\le 6k+2$.
\end{proof}

\begin{proof}[Proof of Theorem \ref{thm:polygon}]
Let $C$ be an arbitrary convex polygon. Denote its vertices in clockwise order by $c_0c_1\dots c_{n-1}$ and its sides by $e_i=c_ic_{i+1}$. We will again add some points to $P$ to define $P'$ and take the Delaunay triangulation of $P'$ with respect to $C$. All the added points lie on the boundary or outside of $C$ and their positions depend on the point set $P$ as follows.

For each $p\in P$ and side $e_i$ (indices are always modulo $n$) of $C$ we do the following.
First draw two lines through $p$ such that the first is parallel to $e_{i-1}$ and the second is parallel to $e_{i+1}$.
These intersect the supporting line $l_i$ of $e_i$ in two points, $p_l$ and $p_r$.
(See Figure \ref{fig:polygon} for an illustration.)
Any homothet $C'$ of $C$ that intersects $l_i$ and has $p$ on its boundary  contains a point of the $\overline{p_lp_r}$ segment.
(Here we allow $C'$ to contain more points in its interior.)
Take a $C'$ for which the length of the intersection of (the closure of) $C'$ and $l_i$ is minimal and denote this minimum by $\epsilon$.
It is easy to see that $\epsilon$ is well-defined and positive.

Also observe that $|p_lp_r|/\epsilon$ depends only on $C$ and $i$ and is independent of the position of $p$, since translating $p$ parallel to $e_i$ only shifts the $pp_lp_r$ triangle with the same quantity, while moving $p$ perpendicularly to $e_i$ only scales $pp_lp_r$, thus scaling both $\epsilon$ and $|p_lp_r|$ by the same value.

Now put $N_i=\lfloor |p_lp_r|/\epsilon\rfloor$ evenly spaced points on the segment $\overline{p_lp_r}$, so that the distance between any two of them is less than $\epsilon$. Moreover, add one point very close to $p_l$ and another very close to $p_r$ onto $l_i$ but not on $\overline{p_lp_r}$, such that the distances from them to the next point is still less than $\epsilon$.
We add these $N_i+2$ points to $P'$.
Since the distance between any two of them is less than $\epsilon$, any homothet of $C$ with $p$ on its boundary and intersecting $l_i$ contains one of the just added $N_i+2$ points in its interior by the definition of $p_l, p_r$ and $\epsilon$.
The number of points we added depends only on $C$ and $i$, not on $p$.
Repeating this for all points and edges of $C$, we add at most $k\sum_{i=1}^n(N_i+2)=O(k)$ points.
(Here the constant in the $O(.)$ notation depends on $C$.)

To ensure that the outerface would not intersect the interior of $C$, we add $n$ more points to $P'$ which are far, outside of $C$, and form a reflected copy of $C$. This way these $n$ points will be the vertices of the outerface.

In the Delaunay triangulation of $P'$ the inner faces are triangles, with a homothet of $C$ covering each triangle not containing any point of $P'$ in its interior. Further, if any of these homothets would extend beyond $C$, then it would contain an added point in its interior, as each of them touches at least one point $p\in P$.

Using Euler's formula, there are $O(k)$ vertices and thus at most $O(k)$ faces. Each inner face is a triangle which can be covered by a homothet of $C$ not containing any point of $P'$ in its interior.
\end{proof}

\begin{figure}[t]
    \centering
		    \subfigure[Proof of Theorem \ref{thm:polygon}]{\label{fig:polygon}
				\includegraphics[scale=0.76]{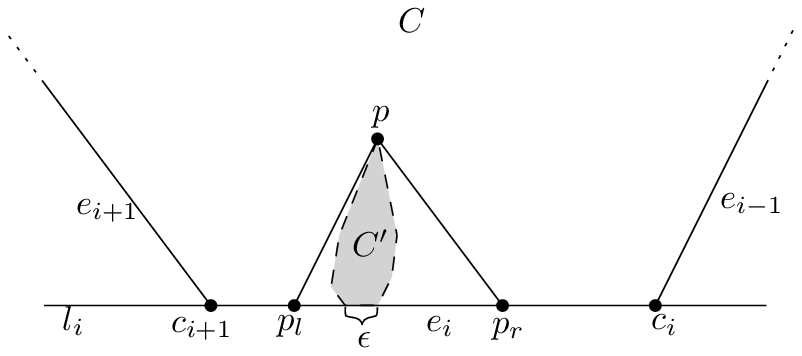}
				}    
    \centering
				\subfigure[Proof of Theorem \ref{thm:quadrangle}]{\label{fig:quadrangle}
				\includegraphics[scale=0.76]{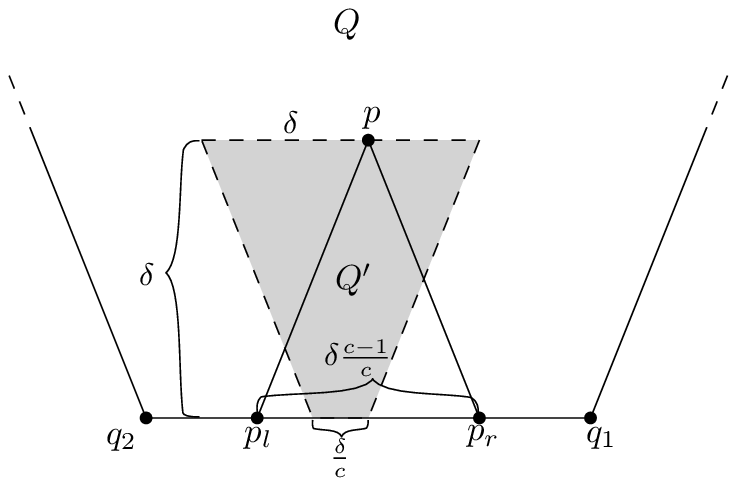}
				}    
\end{figure} 	

Finally, we prove Theorem \ref{thm:quadrangle}, which states that even for a quadrangle we may need many points.
For that we basically prove that while by the above upper bound states that at most $k\sum_{i=1}^n(N_i+2)$ copies are enough to cover $C$, also at least $k\min_i N_i$ copies are necessary to cover $C$.

\begin{proof}[Proof of Theorem \ref{thm:quadrangle}]
Given $c>1$, let $Q$ be a symmetric trapezoid with vertices $q_1,q_2,q_3,q_4$ in clockwise order, with two horizontal edges, the bottom edge, $\overline{q_1q_2}$ has length $1/c$, while the length of the top edge, $\overline{q_3q_4}$, and the height of $Q$ are both equal to $1$. We show that for $Q$ we have $f(k)\ge (c-1)k$. 

Put a point $p$ very close to the bottom edge of $Q$, say the distance of $p$ from the bottom edge is $\delta$. We define $p_l$ and $p_r$ as in the previous proof (see Figure \ref{fig:quadrangle}). Evidently, in the self-cover of $Q$, the points of $p_lp_r$ can only be covered by homothetic copies of $Q$ whose upper edge touches or is below $p$, thus have height at most $\delta$. The length of the top edge of such a $Q'$ is thus also at most $\delta$ and thus the length of its bottom edge is at most $\delta/c$. We also know that $|p_lp_r|=\delta(1-1/c)$. Thus to cover the points of $p_lp_r$ we need at least $\frac{\delta(1-1/c)}{\delta/c}=c-1$ such homothets.

Now if instead of one, we put $k$ points very close to the bottom edge, but far from each other, then for each point we need $c-1$ homothets to cover the respective segment on the boundary of $Q$, thus altogether we need at least $(c-1)k$ homothets, as claimed.
\end{proof}

\section{Remarks and open problems}

Since the first version of this paper, the framework developed here was already successfully applied by Cardinal et al.\ \cite{colorful2} and they obtained that for an octant any finite set of points can be colored with $k$ colors such that any translate of the octant with at least $m_k= 12\cdot (23)^{\lceil \log k\rceil-1}\le 12(k-1)\cdot k^{\log 23}=O(k^{5.53})$ points contains all $k$ colors.
As a corollary, they also obtained that any $m_k$-fold covering of the plane with homothets of a triangle can be decomposed into $k$ coverings.
It was also shown recently by Kov\'acs \cite{K13} that this result cannot be extended to other polygons, as for any $m$ and any non-triangle polygon $C$ there is an $m$-fold covering of the plane with homothets of $C$ that does not decompose into $2$ coverings.

The main question left open is whether for every self-coverable family $m_k=O(k)$, that is, is it true that any finite set of points can be colored with $k$ colors such that any member of the family with $m_k=O(k)$ points contains all $k$ colors if there is an $m$ such that any finite set of points can be colored with two colors such that any member of the family with $m$ points contains both colors?

Similar questions are also open for the dual problem about cover-de\-com\-po\-si\-tion, for more see \cite{PPT11} or \cite{P10}.\\

We would like to thank the anonymous referees for their useful suggestions about how to improve our paper.

\end{document}